\documentclass{article}
\usepackage{graphicx} 
\usepackage{amsfonts}
\usepackage{amsmath}
\usepackage{authblk}
\usepackage{url}

\usepackage{amsmath, amsthm, amssymb}
\newtheorem{theorem}{Theorem}[section]

\title{On the maximal measure of a spherical set avoiding solutions to $x+y+z = 0$}
\author[1]{Ákos Dúcz}

\affil[1]{HUN-REN Alfréd Rényi Institute of Mathematics, Budapest, Hungary}

\date{June 2026}

\begin{document}

\maketitle

\begin{abstract}
    We prove that the maximal normalized surface measure of a spherical set in $d$ dimensions avoiding solutions to $x+y+z = 0$ approaches $1/2$ as $d$ goes to infinity. This gives a partial answer to a question of Bukh, who conjectured $1/2$ to be the optimal bound for all $d \geq 3$ \cite{BukhProblems}.
\end{abstract}

\section{Proof of the main theorem}

First, we need the following generalization of Mantel's famous theorem \cite{Mantel1907}: 

\begin{theorem}[\cite{BrownHarary1970ExtremalDigraphs}]
Let \(D\) be a loopless digraph on \(n\) vertices, with no parallel arcs in the same direction but with opposite arcs allowed. If \(D\) contains no directed triangle, then
\[
|E(D)|\le \left\lfloor \frac{n^2}{2}\right\rfloor ,
\]
where $E(D)$ denotes the set of arcs of $D$.
\end{theorem}

For the sake of completeness, a proof of this theorem is provided in the appendix.
Let us now prove our main theorem:

\begin{theorem}
Let $d\ge 3$, and let $A\subset S^{d-1}$ be measurable. Suppose that there are
no $x,y,z\in A$ such that
\[
    x+y+z=0.
\]
Then
\[
    \sigma(A)
    \le
    \frac{\left\lfloor d^2/2\right\rfloor}{d(d-1)},
\]
where $\sigma$ denotes normalized surface measure on $S^{d-1}$.
In particular,
\[
    \sigma(A)\le \frac{d}{2(d-1)}.
\]
\end{theorem}

\begin{proof}
Let $e_1,\dots,e_d$ be the standard basis of $\mathbb R^d$, and consider the
set
\[
    V=\left\{ \frac{e_i-e_j}{\sqrt 2} : i,j\in\{1,\dots,d\},\ i\ne j \right\}
    \subset S^{d-1}.
\]
We think of the point $(e_i-e_j)/\sqrt2$ as the directed edge $i\to j$.

Let $U$ be a random element of $SO(d)$, chosen according to Haar probability
measure. For each such $U$, define a directed graph $D_U$ on vertex set
$\{1,\dots,d\}$ by putting
\[
    i\to j \in E(D_U)
    \quad\Longleftrightarrow\quad
    U\left(\frac{e_i-e_j}{\sqrt2}\right)\in A.
\]

We claim that $D_U$ contains no directed triangle. Indeed, if
\[
    i\to j,\qquad j\to k,\qquad k\to i
\]
were a directed triangle in $D_U$, then the three corresponding points of $A$
would be
\[
    U\left(\frac{e_i-e_j}{\sqrt2}\right),\qquad
    U\left(\frac{e_j-e_k}{\sqrt2}\right),\qquad
    U\left(\frac{e_k-e_i}{\sqrt2}\right).
\]
Their sum is
\[
    U\left(
        \frac{e_i-e_j+e_j-e_k+e_k-e_i}{\sqrt2}
    \right)
    =0,
\]
contradicting the assumption on $A$.

Therefore, by Theorem 1.1,
\[
    |E(D_U)|\le \left\lfloor \frac{d^2}{2}\right\rfloor
\]
for every $U\in SO(d)$.

On the other hand, averaging over all rotations gives
\[
    \mathbb E_U |E(D_U)|
    =
    \sum_{i\ne j}
    \mathbb P\left[
        U\left(\frac{e_i-e_j}{\sqrt2}\right)\in A
    \right].
\]
For each fixed $i\ne j$, the point
\[
    U\left(\frac{e_i-e_j}{\sqrt2}\right)
\]
is uniformly distributed on $S^{d-1}$. Hence each summand is equal to
$\sigma(A)$, and since there are $d(d-1)$ ordered pairs $(i,j)$ with $i\ne j$,
we obtain
\[
    \mathbb E_U |E(D_U)| = d(d-1)\sigma(A).
\]
Combining this with the pointwise upper bound on $|E(D_U)|$ gives
\[
    d(d-1)\sigma(A)
    \le
    \left\lfloor \frac{d^2}{2}\right\rfloor .
\]
Thus
\[
    \sigma(A)
    \le
    \frac{\left\lfloor d^2/2\right\rfloor}{d(d-1)}.
\]
Finally,
\[
    \frac{\left\lfloor d^2/2\right\rfloor}{d(d-1)}
    \le
    \frac{d^2/2}{d(d-1)}
    =
    \frac{d}{2(d-1)}.
\]
This completes the proof.
\end{proof}

We also note that $\sigma(A) = 1/2$ is trivially achievable in all dimensions, therefore Theorem 1.2 is asymptotically sharp.

\section{Acknowledgements}

The author was supported by grant NKFIH-153165.

\bibliographystyle{alpha}
\bibliography{ref}

\section{Appendix}

\begin{proof}[Proof of Theorem 1.1]
Write
\[
d^+(v)=|\{u:v\to u\}|,\qquad d^-(v)=|\{u:u\to v\}|,
\qquad t(v)=d^+(v)+d^-(v).
\]
We first prove the following claim.

\medskip
\noindent
\textbf{Claim.} There is a vertex \(v\) such that
\[
t(v)\le 
\begin{cases}
n, & n \text{ even},\\
n-1, & n \text{ odd}.
\end{cases}
\]

Indeed, since \(D\) has no directed triangle, for every arc \(x\to y\) we have
\[
N^-(x)\cap N^+(y)=\varnothing,
\]
because any \(z\in N^-(x)\cap N^+(y)\) would give
\[
z\to x\to y\to z.
\]
Hence
\[
d^-(x)+d^+(y)\le n.
\tag{1}
\]

Suppose first that \(t(v)\ge n+1\) for every vertex \(v\). If \(x\to y\), then by (1),
\[
d^+(x)=t(x)-d^-(x)\ge n+1-d^-(x)>d^+(y).
\]
Thus along every arc the outdegree strictly decreases. This is impossible on a directed cycle, and in particular forbids two-cycles. Hence \(D\) is acyclic, so
\[
|E(D)|\le \binom n2.
\]
But \(t(v)\ge n+1\) for every \(v\) gives
\[
2|E(D)|=\sum_v t(v)\ge n(n+1),
\]
so
\[
|E(D)|\ge \frac{n(n+1)}2>\binom n2,
\]
a contradiction. Therefore some vertex has \(t(v)\le n\).

It remains only to strengthen this by one when \(n\) is odd. Suppose, for contradiction, that \(n\) is odd and
\[
t(v)\ge n
\]
for every vertex \(v\). Then for every arc \(x\to y\), (1) gives
\[
d^+(x)=t(x)-d^-(x)\ge n-d^-(x)\ge d^+(y).
\tag{2}
\]
So arcs never go from smaller outdegree to larger outdegree.

Partition \(V(D)\) into classes of equal outdegree. Consider one such class \(C\), with \(|C|=s\). If \(x,y\in C\) and \(x\to y\), then \(d^+(x)=d^+(y)\), so equality holds in (1). Hence \(d^-(x)+d^+(y)=n\). If \(y\to x\) were missing, then both \(x\) and \(y\) would be missing from
\[
N^-(x)\cup N^+(y),
\]
contradicting \(d^-(x)+d^+(y)=n\) and \(N^-(x)\cap N^+(y)=\varnothing\). Thus every arc inside \(C\) is bidirected.

Let \(H\) be the ordinary graph on \(C\) whose edges are these bidirected pairs. Since a triangle in \(H\) would give a directed triangle in \(D\), the graph \(H\) is triangle-free.

For every \(v\in C\), vertices outside \(C\) contribute at most \(n-s\) to \(t(v)\), while each neighbor in \(H\) contributes \(2\). Therefore
\[
t(v)\le n-s+2d_H(v).
\]
Since \(t(v)\ge n\), we get
\[
d_H(v)\ge \frac{s}{2}
\]
for every \(v\in C\).

But \(H\) is triangle-free. By Mantel's theorem \cite{Mantel1907},
\[
e(H)\le \frac{s^2}{4}.
\]
On the other hand,
\[
e(H)=\frac12\sum_{v\in C}d_H(v)\ge \frac12\cdot s\cdot \frac{s}{2}
=\frac{s^2}{4}.
\]
Thus equality holds. In particular \(s\) must be even; if \(s\) were odd, then \(d_H(v)\ge \lceil s/2\rceil\) for all \(v\), giving \(e(H)>s^2/4\), impossible.

So every equal-outdegree class has even size. Their sizes sum to \(n\), contradicting that \(n\) is odd. Hence, when \(n\) is odd, some vertex satisfies \(t(v)\le n-1\). This proves the claim.

\medskip

We now prove the theorem by induction on \(n\). The case \(n=1\) is trivial. Let \(v\) be a vertex given by the claim. Then \(D-v\) is also directed-triangle-free, so by induction,
\[
|E(D-v)|\le \left\lfloor \frac{(n-1)^2}{2}\right\rfloor .
\]
If \(n\) is even, then \(t(v)\le n\), and therefore
\[
|E(D)|\le \left\lfloor \frac{(n-1)^2}{2}\right\rfloor+n
=\frac{n^2}{2}.
\]
If \(n\) is odd, then \(t(v)\le n-1\), and therefore
\[
|E(D)|\le \left\lfloor \frac{(n-1)^2}{2}\right\rfloor+n-1
=\left\lfloor \frac{n^2}{2}\right\rfloor .
\]
Thus in all cases,
\[
|E(D)|\le \left\lfloor \frac{n^2}{2}\right\rfloor .
\]
\end{proof}

\end{document}